\documentclass[10pt]{amsart}

\makeatletter
\def\blfootnote{\gdef\@thefnmark{}\@footnotetext}
\makeatother

\usepackage{epsfig}
\usepackage{graphics}
\usepackage{dcpic, pictexwd}

\usepackage[colorlinks=true,
                    linkcolor=blue,
                    urlcolor=blue,
                    citecolor=blue,
                    anchorcolor=blue]{hyperref}
\usepackage{mathtools}
\usepackage{amsmath,amssymb}

\theoremstyle{plain}

\newtheorem*{theorem*}{Theorem}
\newtheorem*{thma}{Theorem A}
\newtheorem*{thmb}{Theorem B}

\newtheorem{theorem}{Theorem}[section]

\newtheorem{lemma}[theorem]{Lemma}

\theoremstyle{remark}
\newtheorem{remark}[theorem]{Remark}
\theoremstyle{Acknowledgments}

\theoremstyle{definition}

 \def\Z{{\mathbb{Z}}}

\def\mod{{\rm Mod}}

 \def\Sp{{\rm Sp}}

\begin{document}
\blfootnote{\textup{2000} \textit{Mathematics Subject Classification}:
57M07, 20F05, 20F38}
\newenvironment{prooff}{\medskip \par \noindent {\it Proof}\ }{\hfill
$\square$ \medskip \par}
    \def\sqr#1#2{{\vcenter{\hrule height.#2pt
        \hbox{\vrule width.#2pt height#1pt \kern#1pt
            \vrule width.#2pt}\hrule height.#2pt}}}
    \def\square{\mathchoice\sqr67\sqr67\sqr{2.1}6\sqr{1.5}6}
\def\pf#1{\medskip \par \noindent {\it #1.}\ }
\def\endpf{\hfill $\square$ \medskip \par}
\def\demo#1{\medskip \par \noindent {\it #1.}\ }
\def\enddemo{\medskip \par}
\def\qed{~\hfill$\square$}

 \title[On the Torsion Generators ] {On the Torsion Generators of the Mapping Class Groups}

\author[T{\"{u}}l\.{i}n Altun{\"{o}}z,       Mehmetc\.{i}k Pamuk, and O\u{g}uz Y{\i}ld{\i}z ]{T{\"{u}}l\.{i}n Altun{\"{o}}z,    Mehmetc\.{i}k Pamuk, and O\u{g}uz Y{\i}ld{\i}z}

\address{Faculty of Engineering, Ba\c{s}kent University, Ankara, Turkey} 
\email{tulinaltunoz@baskent.edu.tr} 
\address{Department of Mathematics, Middle East Technical University,
 Ankara, Turkey}
 \email{mpamuk@metu.edu.tr}
 \address{Department of Mathematics, Middle East Technical University,
 Ankara, Turkey}
  \email{oguzyildiz16@gmail.com}


\begin{abstract} 
We study torsion generators for the (extended) mapping class group $\mod(\Sigma_g)$ ($\mod^{*}(\Sigma_{g})$) of a closed connected orientable surface of genus $g$.  
We show that for every $g\geq 14$, $\mod(\Sigma_{g})$ can be generated by two torsion elements of order $g+1$ if $g$ is even, and of orders $g+1$ and $\frac{g+1}{2}$ if $g$ is odd.  Also for $g\geq 16$, $\mod(\Sigma_g)$ can be generated by two torsion elements  of orders $g+1$ if $g+1$ is not divisible by $3$, and of orders $g+1$ and $\frac{g+1}{3}$ if $g+1$ is divisible by $3$. Similarly, we  obtain two torsion elements generating $\mod^{*}(\Sigma_{g})$.

\end{abstract}
\maketitle
  \setcounter{secnumdepth}{2}
 \setcounter{section}{0}
 
\section{Introduction}

Let $\Sigma_{g}$ denote a connected orientable surface of genus $g$. The mapping class group of $\Sigma_g$ is defined as the group of isotopy classes of orientation preserving diffeomorphisms and it is denoted by $\mod(\Sigma_g)$.  It is a classical result that $\mod(\Sigma_g)$ is generated by finitely many Dehn twists about nonseparating simple closed curves~\cite{de,H,l3}.

One of the main themes of our paper is to generate $\mod(\Sigma_g)$ with torsion elements of certain order.  In $1996$, Wajnryb~\cite{w} showed that $\mod(\Sigma_g)$ can be generated by two elements given as a product of Dehn twists. As the group is not abelian, this is the smallest possible. Korkmaz~\cite{mk2} improved this result by first showing that one of the two generators can be taken as a Dehn twist and the other as a torsion element. 
Korkmaz also proved that $\mod(\Sigma_g)$ can be generated by two torsion elements of order $4g+2$. Recently, the third author showed that $\mod(\Sigma_g)$ is generated by two torsions of small orders~\cite{y1}. McCarthy and Papadopoulus~\cite{mp}  showed that $\mod(\Sigma_g)$ can be generated 
by infinitely many conjugates of a single involution (element of order two) for $g\geq 3$.     In terms of generating by finitely many involutions, Luo~\cite{luo} showed that any Dehn twist about a nonseparating simple closed curve 
can be written as a product six involutions, which in turn implies that $\mod(\Sigma_g)$ can be generated by $12g+6$ involutions.  
Brendle and Farb~\cite{bf} obtained a generating set of six involutions for $g\geq3$. Following their work, Kassabov~\cite{ka} showed that 
$\mod(\Sigma_g)$ can be generated by four involutions if $g\geq7$.  Recently, Korkmaz~\cite{mk1} showed that $\mod(\Sigma_g)$ is generated by three involutions 
if $g\geq8$ and four involutions if $g\geq3$. The third author improved these results by showing that this group can be generated by three involutions if $g\geq6$~\cite{y2}.  The first of our main results in this direction is the following theorem (see Theorems \ref{lem1} and \ref{lem2}). 

\begin{thma}\label{thma}
For every $g\geq 14$, $\mod(\Sigma_{g})$ can be generated by two torsion elements of order $g+1$ if $g$ is even, and of orders $g+1$ and $\frac{g+1}{2}$ if $g$ is odd.  Also for $g\geq 16$,  $\mod(\Sigma_g)$ can be generated by two torsion elements  of orders $g+1$ if $g+1$ is not divisible by $3$, and of orders $g+1$ and $\frac{g+1}{3}$ if $g+1$ is divisible by $3$.
\end{thma}

\begin{remark}
Since there is a surjective homomorphism from $\mod(\Sigma_{g})$ onto the symplectic group $\Sp(2g, \Z)$, the same conclusions can also be made for the symplectic group.  
\end{remark}

Recall that the extended mapping class group, denoted by $\mod^{*}(\Sigma_{g})$,  is defined to be the group of isotopy classes of all self-diffeomorphisms of $\Sigma_{g}$, including orientation reversing self-diffeomorphisms. 
The mapping class group $\mod(\Sigma_{g})$ is an index two normal subgroup of $\mod^{*}(\Sigma_{g})$.

In the second part of the paper, we try to extend our results about $\mod(\Sigma_{g})$ to $\mod^{*}(\Sigma_{g})$.  First let us review the previous results about generating  $\mod^{*}(\Sigma_{g})$. Korkmaz~\cite{mk2}, proved that $\mod^{*}(\Sigma_{g})$ can be generated by two elements, one of which is a Dehn twist.
Moreover, it follows from~\cite[Theorem $14$]{mk2} that $\mod^{*}(\Sigma_{g})$ can be generated by three torsion elements for $g\geq1$. 
Also, Du~\cite{du1, du2} showed that $\mod^{*}(\Sigma_{g})$ can be generated by two torsion elements of order $2$ and $4g+2$ for $g\geq 3$.
In terms of involution generators,  as it contains nonabelian free groups, the minimal number of involution generators is three and 
Stukow~\cite{st} proved that $\mod^{*}(\Sigma_{g})$ can be generated by three involutions for $g\geq1$.  Our main result on this subject can be summarized as in the following theorem (see Theorems \ref{lem11} and \ref{lem12}).  

\begin{thmb}\label{thmb}
If $g\geq 14$, then $\mod^{*}(\Sigma_g)$ can be generated by two torsion elements.  The order of one of them is $g+1$ or $2g+2$ when $g$ is odd or even, respectively. The order of the other generator is $g+1$ or $\frac{g+1}{2}$ when $g$ is even or odd, respectively.

Also for $g\geq 16$, $\mod^{*}(\Sigma_g)$ can be generated by two elements. 
The order of one of the generators is $g+1$ or $2g+2$  when $g$ is odd or even, respectively.  The order of the second generator can be summarized as below:
\[ 
\begin{cases} 
      g+1 & \textrm{if}~ 3\nmid g+1 ~\textrm{and}~ g ~\textrm{is odd,} \\
       2g+2 & \textrm{if}~ 3\nmid g+1 ~\textrm{and}~ g ~\textrm{is even,} \\
      \frac{g+1}{3} & \textrm{if}~ 6\mid g+1, \\
       \frac{2g+2}{3} & \textrm{if}~ 6\mid g+1 ~\textrm{and}~ 6\nmid g+1.
   \end{cases}
\]
\end{thmb}

The paper is organized as follows: In Section~\ref{S2}, we quickly provide the necessary background on mapping class groups and state a useful result from algebra. In Section~\ref{S3}, we state our results on the torsion generating sets of $\mod(\Sigma_{g})$.  In Section~\ref{S4}, we extend our results about $\mod(\Sigma_{g})$ to $\mod^{*}(\Sigma_{g})$. 


\par  
\section{Background and Results on Mapping Class Groups} \label{S2}

  Throughout the paper we do not distinguish a 
 diffeomorphism from its isotopy class. For the composition of two diffeomorphisms, we
use the functional notation; if $f$ and $g$ are two diffeomorphisms, then
the composition $fg$ means that $g$ is applied first and then $f$.

For a simple closed curve $a$ on $\Sigma_{g}$, we denote the right-handed  
Dehn twist $t_a$ about $a$ by the corresponding capital letter $A$.
Let us also remind the following basic facts of Dehn twists that we use frequently: Let $a$ and $b$ be two
simple closed curves on $\Sigma_{g}$ and $f\in \mod(\Sigma_{g})$.
\begin{itemize}
\item  If $a$ and $b$ are disjoint, then $AB=BA$ (\textit{Commutativity}).
\item If $f(a)=b$, then $fAf^{-1}=B$ (\textit{Conjugation}).
\end{itemize}

\begin{figure}[hbt!]
\begin{center}
\scalebox{0.25}{\includegraphics{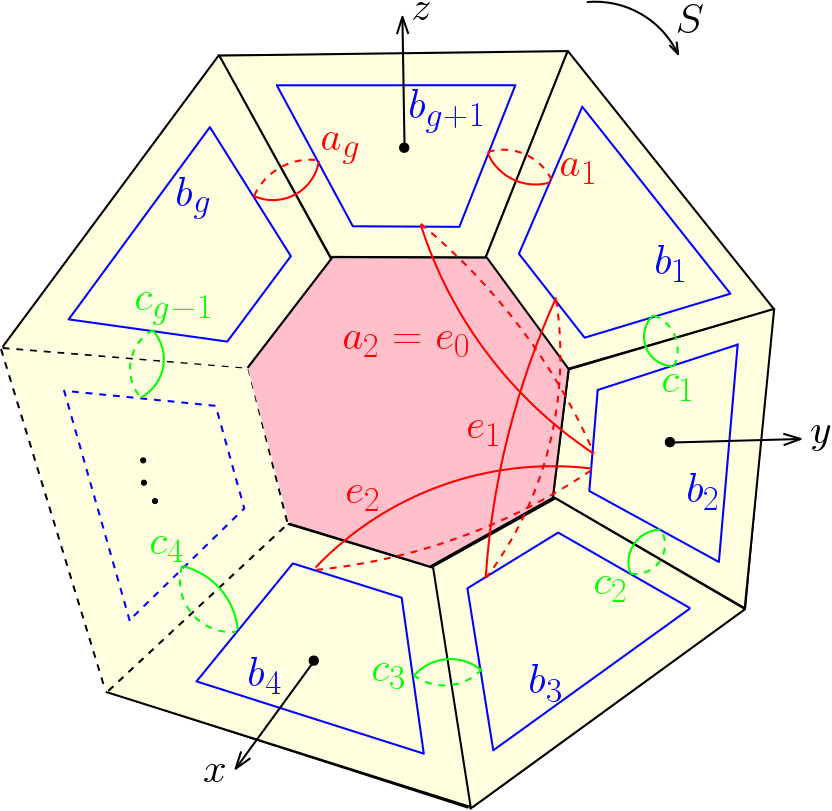}}
\caption{The rotation $S$ and the curves $a_1,a_2,a_g, b_i,c_i, e_i$.}
\label{rotation}
\end{center}
\end{figure}

\begin{figure}[hbt!]
\begin{center}
\scalebox{0.3}{\includegraphics{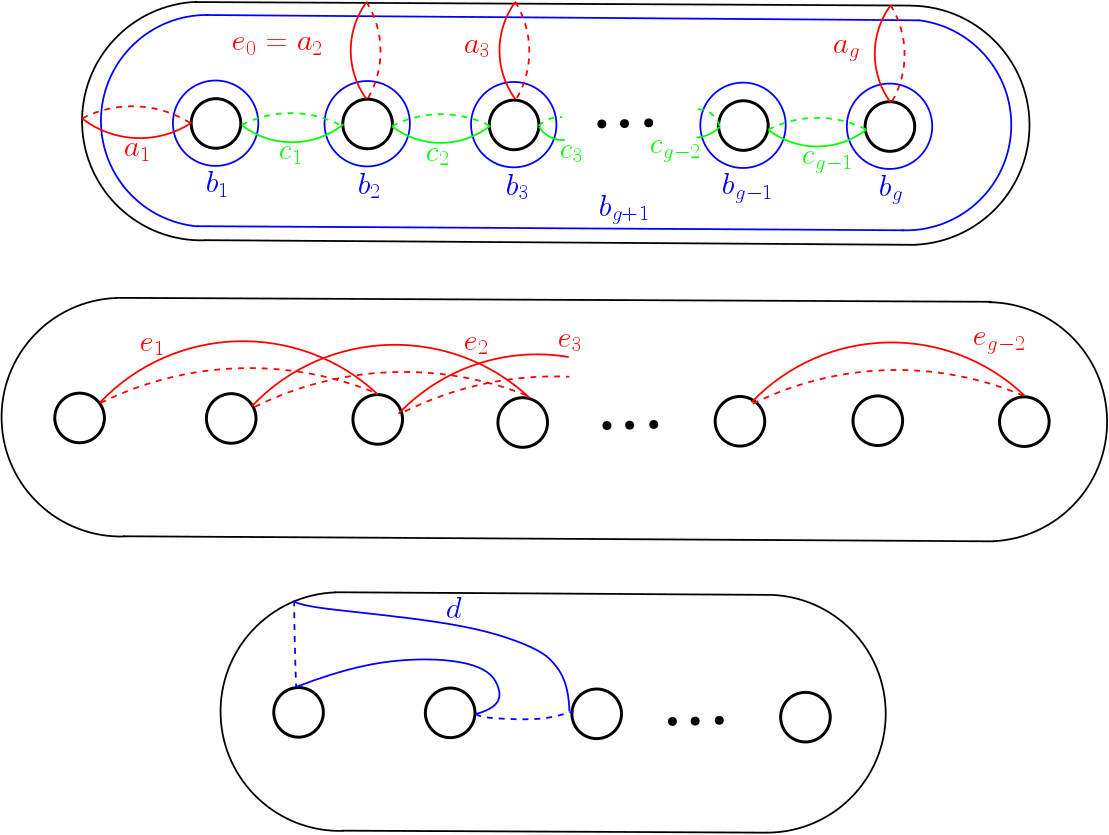}}
\caption{The curves $a_1,a_2,a_3,a_g, b_i,c_i, e_i$ and $d$.}
\label{curves}
\end{center}
\end{figure}

We consider the embedding of $\Sigma_{g}$ into $\mathbb{R}^{3}$ 
as depicted in Figure~\ref{rotation}.  This surface is the boundary of the solid handlebody consisting of two thickened $(g+1)$-gons, one placed in the $x=-1$ and the other in the $x=1$ plane, and $(g+1)$ solid handles joining their corresponding vertices (see also \cite[Figure 2]{bk}).  Observe also  that $\Sigma_{g}$ is embedded into $\mathbb{R}^{3}$ in such a way that it is invariant under the rotation, $S$, by $\dfrac{2\pi}{g+1}$.

Note that the rotation $S$ satisfies the following:
\begin{itemize}
\item [(i)] $S(a_g)=a_{1}$, $S(a_1)=c_{1}$,
\item [(ii)] $S(b_i)=b_{i+1}$ for $i=1,\ldots,g$, 
\item [(iii)]  $S(c_i)=c_{i+1}$ for $i=1,\ldots, g-2$, $S(c_{g-1})=a_g$ and
\item [(iii)] $S(e_i)=e_{i+1}$ for $i=0,\ldots, g-1$, $S(e_g)=e_0$.
\end{itemize}

Let us finish this section by stating a result from algebra \cite[Lemma~21]{y1}, which we use to determine the orders of torsion elements in our generating sets.
\begin{lemma}\label{order}
If $R$ is an element of order $k$ in a group $G$ and if $x$ and $y$ are elements in $G$ satisfying $RxR^{-1}= y$, then the order of $Rxy^{-1}$ is also $k$.   
\end{lemma}

\section{Torsion generators for $\mod(\Sigma_{g})$}\label{S3}
Let us start this section by giving a set of generators consisting of the rotation $S$ inspired by the set of generators given by Korkmaz~\cite[Theorem~$5$]{mk1}.

\begin{theorem}\label{thm1}
If $g\geq3$, then the mapping class group $\mod(\Sigma_g)$ is generated by the four elements $S$, $A_1A_{2}^{-1}$, $B_1B_{2}^{-1}$, $C_1C_{2}^{-1}$.
\end{theorem}
\begin{proof}
First, we recall that by \cite{H}, 
$
\lbrace A_1, A_2, B_1, \ldots, B_g, C_1, \ldots C_{g-1} \rbrace
$
generates $\mod(\Sigma_g)$.  We refer to the elements of this set as Humphries generators.

Let $G$ be the subgroup of $\mod(\Sigma_g)$ generated by the set
\[
\lbrace S,  A_1A_{2}^{-1}, B_1B_{2}^{-1}, C_1C_{2}^{-1} \rbrace.
\]
It is easy to see that $S$ maps the curves $(b_1,b_2)$ to $(b_2,b_3)$. Thus, we get the element
\[
B_2B_3^{-1}=S(B_1B_2^{-1})S^{-1}\in G.
\]
One can also check that the diffeomorphism $(B_2B_3^{-1})(A_2A_1^{-1})$ satisfies
\[
(B_2B_3^{-1})(A_2A_1^{-1})(b_2,b_3)=(a_2,b_3).
\]
As a result we obtain,  
\[
A_2B_3^{-1}=\big((B_2B_3^{-1})(A_2A_1^{-1})\big)(B_2B_3^{-1})\big((B_2B_3^{-1})(A_2A_1^{-1})\big)^{-1}\in  G,
\]
because of the fact that each factor on the right hand side is contained in $G$.\\
From these we have
\begin{eqnarray*}
A_1B_2^{-1}&=&(A_1A_2^{-1})(A_2B_3^{-1})(B_3B_2^{-1})\in G \textrm{ and so}\\
A_1B_1^{-1}&=&(A_1B_2^{-1})(B_2B_1^{-1})\in G.
\end{eqnarray*}
Since $S(a_1,b_1)=(c_1,b_2)$, we have
\[
C_1B_2^{-1}=S(A_1B_1^{-1})S^{-1}\in G.
\]
Using this element, we get
\[
B_1C_1^{-1}=(B_1B_2^{-1})(B_2C_1^{-1})\in G,
\]
which implies that $B_1C_1^{-1}, B_2C_2^{-1},\ldots, B_{g-1}C_{g-1}^{-1}$ are all contained in $G$ by the action of $S$.

The following elements are also contained in $G$:
\begin{eqnarray*}
A_1C_1^{-1}&=&(A_1B_2^{-1})(B_2C_1^{-1}),\\
C_1A_2^{-1}&=&(C_1A_1^{-1})(A_1A_2^{-1}),\\
C_2A_1^{-1}&=&(C_2C_1^{-1})(C_1A_1^{-1}) \textrm{ and}\\
C_2B_1^{-1}&=&(C_2A_1^{-1})(A_1B_1^{-1}).
\end{eqnarray*}

Since $S(a_2,c_1)=(e_1,c_2)$, we have
\[
E_1C_2^{-1}=S(A_2C_1^{-1})S^{-1}\in G.
\]

Moreover, the diffeomorphism
\[
\Gamma=(B_1C_2^{-1})(C_1C_2^{-1})(A_1C_2^{-1})(B_1C_2^{-1})\in G,
\]
and it maps $(e_1,c_2)$ to $(d,c_2)$. This implies that
\[
DC_2^{-1}=\Gamma(E_1C_2^{-1})\Gamma^{-1}\in G.
\]

The subgroup $G$ contains 
\[
DA_1^{-1}=(DC_2^{-1})(C_2A_1^{-1}).
\]

The following relation holds in $\mod(\Sigma_g)$ by the lantern relation:
\[
A_1C_1C_2A_3=A_2E_1D,
\]
where the curves are depicted in Figure~\ref{curves}.  Note that, this can be rewritten as
\[
A_3=(A_2C_1^{-1})(E_1C_2^{-1})(DA_1^{-1}),
\]
which is clearly contained in $G$. Hence, it follows from
\[
A_3(B_3B_2^{-1})(a_3)=b_3,
\]
that we get 
\[
B_3=\big(A_3(B_3B_2^{-1})\big)(A_3)\big(A_3(B_3B_2^{-1})\big)^{-1}\in G,
\]
which implies $B_1,B_2,\ldots, B_{g+1}$ are all contained in $G$ by the action of $S$. Moreover, we obtain 
\begin{eqnarray*}
C_1&=&(C_1B_1^{-1})B_1\in G, \\
A_1&=&(A_1B_1^{-1})B_1\in G\textrm{ and so}\\
A_2&=&(A_2A_1^{-1})A_1\in G.\
\end{eqnarray*}

Also, by conjugating $C_1$ with powers of $S$, we can see that $C_2,\ldots, C_{g-1}$ are in $G$. This finishes the proof since all Humphries generators are contained in $G$.
\end{proof}

In the following theorem, we present a generating set for $\mod(\Sigma_g)$ consisting of two elements of certain orders.

\begin{theorem}\label{lem1}
If $g\geq 14$, then $\mod(\Sigma_g)$ can be generated by the two elements $S$ and $S^2 A_g B_6A_2E_2^{-1}B_8^{-1}C_1^{-1}$.  The order of $S$ is $g+1$, and the order of $S^2 A_gB_6A_2E_2^{-1}B_8^{-1}C_1^{-1}$ is $g+1$ or $\frac{g+1}{2}$ when $g$ is even or odd, respectively.
\end{theorem}

\begin{proof}
Let $H$ be the subgroup of  $\mod(\Sigma_g)$ generated by $S$ and $S^2F_1$, where  $F_1:=A_gB_6A_2E_2^{-1}B_8^{-1}C_1^{-1}$.  It is clear that $F_1\in H$. We then have the following element, we call $F_2$, in $H$.
\[
F_2:=S^2F_1S^{-2}=C_1B_8E_2E_4^{-1}B_{10}^{-1}C_3^{-1}.
\]
One can check that $F_1F_2^{-1}(a_g,b_6,a_2,e_2,b_8,c_1)=(a_g,e_4,a_2,e_2,b_8,c_1)$, which gives rise to the following element
\[
F_3:=(F_1F_2^{-1})F_1(F_1F_2^{-1})^{-1}=A_gE_4A_2E_2^{-1}B_8^{-1}C_1^{-1}\in H.
\]
From these we get $F_1F_3^{-1}=B_6E_4^{-1}\in H$, which implies
\begin{eqnarray*}
B_4E_2^{-1}=S^{-2}(F_1F_3^{-1})S^2\in H.
\end{eqnarray*}
We also obtain
\[
F_4:=S^5F_1S^{-5}=C_4B_{11}E_5E_7^{-1}B_{13}^{-1}C_6^{-1}\in H.
\]
Moreover, it can be verified that $F_1F_4^{-1}(a_g,b_6,a_2,e_2,b_8,c_1)=(a_g,c_6,a_2,e_2,b_8,c_1)$, which gives the following element
\[
F_5:=(F_1F_4^{-1})F_1(F_1F_4^{-1})^{-1}=A_gC_{6}A_2E_2^{-1}B_{8}^{-1}C_1^{-1}\in H.
\]
\begin{remark}
 If $g=13$, then not only the curves $b_6$ and $c_6$ but also the curves $a_g$ and $b_{13}$ would intersect.  As a result, we do not have the expression stated above for $F_5$ if $g=13$.  Similarly, for $g<13$, there are unwanted intersections which prevent the applicability of our method.   
\end{remark}

We then have $B_6C_6^{-1}=F_1F_5^{-1}\in H$. As a result, by the action of $S$, we can conclude that $B_1C_1^{-1}, B_2C_2^{-1},\ldots, B_{g-1}C_{g-1}^{-1}$ are all contained in the subgroup $H$.

The subgroup $H$ also contains the elements 
\begin{eqnarray*}
B_gA_g^{-1}&=&S^{-7}(F_1F_5^{-1})S^7 \textrm{and so}\\
F_6:&=&(B_gA_g^{-1})F_1(E_2B_4^{-1})\\
&=&(B_gA_g^{-1})(A_gB_6A_2E_2^{-1}B_8^{-1}C_1^{-1})(E_2B_4^{-1})\\
&=&B_gB_6A_2B_4^{-1}B_8^{-1}C_1^{-1}
\end{eqnarray*}
and
\begin{eqnarray*}
F_7:&=&(B_6E_4^{-1})(S^4F_6S^{-4})\\
&=&(B_6E_4^{-1})(B_3B_{10}E_4B_8^{-1}B_{12}^{-1}C_5^{-1})\\
&=&B_3B_{10}B_6B_8^{-1}B_{12}^{-1}C_5^{-1}\in H.
\end{eqnarray*}
Since $F_6F_7^{-1}(b_g,b_6,a_2,b_4,b_8,c_1)=(b_g,c_5,a_2,b_4,b_8,c_1)$, the conjugation of $F_6$ with $F_6F_7^{-1}$ gives
\[
F_8:=(F_6F_7^{-1})F_6(F_6F_7^{-1})^{-1}=B_gC_5A_2B_4^{-1}B_8^{-1}C_1^{-1}\in H.
\]
From this, we have
\[
B_6C_5^{-1}=F_6F_8^{-1}\in H,
\]
which implies that $B_{2}C_1^{-1}, B_{3}C_2^{-1},\ldots, B_gC_{g-1}^{-1}$ are in $H$ by conjugation by powers of $S$. Hence, we obtain
\begin{eqnarray*}
C_1C_2^{-1}&=&(C_1B_2^{-1})(B_2C_2^{-1})\in H \textrm{ and}\\
B_1B_2^{-1}&=&(B_1C_1^{-1})(C_1B_2^{-1})\in H.
\end{eqnarray*}
By Theorem~\ref{thm1}, it remains to prove that $A_1A_2^{-1}$ is contained in $H$.
\begin{eqnarray*}
A_1B_1^{-1}&=&S^{-1}(C_1B_2^{-1})S\in H,\\
B_2A_2^{-1}&=&S^{-2}(B_4E_2^{-1})S^2 \in H \textrm{ and so}\\
A_1A_2^{-1}&=&(A_1B_1^{-1})(B_1B_2^{-1})(B_2A_2^{-1})\in H,
\end{eqnarray*}
which completes the first part of the proof.

For the order of $S^2 A_gB_6A_2E_2^{-1}B_8^{-1}C_1^{-1}$, first note that $S^2$ maps the curves $a_g$, $b_6$ and $a_2$ to the curves $c_1$, $b_8$ and $e_2$, respectively.  Now, since the order of $S^2$ is $g+1$ when $g$ is even and $\frac{g+1}{2}$ when $g$ is odd, our conclusion about the orders of the generators follows immediately from Lemma~\ref{order}.
\end{proof}

In the following theorem, we present another generating set  consisting of two elements of certain orders.

\begin{theorem}\label{lem2}
If $g\geq 16$, then $\mod(\Sigma_g)$ can be generated by the two elements $S$ and $S^3A_gE_0B_8B_{11}^{-1}E_3^{-1}C_2^{-1}$ of orders $g+1$ and $\frac{g+1}{3}$ or $g+1$ if $g+1$ is divisible by $3$ or not, respectively.
\end{theorem}
\begin{proof}
Let us denote the factorization of Dehn twist $A_gE_0B_8B_{11}^{-1}E_3^{-1}C_2^{-1}$ by $K_1$ and
$K$ be the subgroup of $\mod(\Sigma_g)$ generated by $S$ and $S^3K_1$. Hence, it is easy to see that $K_1\in K$. We have the following element
\[
K_2:=S^3K_1S^{-3}=C_2E_3B_{11}B_{14}^{-1}E_6^{-1}C_2^{-1}\in K.
\]
One can observe that $K_1K_2^{-1}$ sends the curves $(a_g,e_0,b_8,b_{11},e_3,c_2)$ to $(a_g,e_0,e_6,b_{11},e_3,c_2)$. Then we get
\[
K_3:=(K_1K_2^{-1})K_1(K_1K_2^{-1})^{-1}=A_gE_0E_6B_{11}^{-1}E_3^{-1}C_2^{-1}\in K.
\]
From these, the subgroup $K$ contains $K_1K_3^{-1}=B_8E_6^{-1}$, which implies that
\begin{eqnarray*}
B_5E_3^{-1}&=&S^{-3}(B_8E_6^{-1})S^3\in K \textrm{ and so}\\
B_2E_0^{-1}&=&S^{-3}(B_5E_3^{-1})S^3\in K.
\end{eqnarray*}
We then have 
\begin{eqnarray*}
K_4&:=&K_1(E_3B_5^{-1})=A_gE_0B_8B_{11}^{-1}B_5^{-1}C_2^{-1}\in K \textrm{ and} \\ 
K_5&:=&S^2K_4S^{-2}=C_1E_2B_{10}B_{13}^{-1}B_7^{-1}C_4^{-1}\in K.
\end{eqnarray*}
Consider the diffeomorphism $K_4K_5$ mapping the curves $(a_g,e_0,b_8,b_{11},b_5,c_2)$ to $(a_g,e_0,b_8,b_{11},c_4,c_2)$, which gives the element
\[
K_6:=(K_4K_5)K_4(K_4K_5)^{-1}=A_gE_0B_8B_{11}^{-1}C_4^{-1}C_2^{-1}\in K.
\]
Hence, $K_4^{-1}K_6=B_5C_4^{-1}$ is contained in $K$. And so, we obtain
\begin{eqnarray*}
B_2C_1^{-1}&=&S^{-3}(B_5C_4^{-1})S^3\in K \textrm{ and}\\
B_1A_1^{-1}&=&S^{-1}(B_2C_1^{-1})S\in K.
\end{eqnarray*}
Moreover, we have
\begin{eqnarray*}
    K_7&:=&(C_1B_2^{-1})(B_2E_0^{-1})K_6\\
    &=&(C_1E_0^{-1})A_gE_0B_8B_{11}^{-1}C_4^{-1}C_2^{-1}\\
     &=& A_gC_1B_8B_{11}^{-1}C_4^{-1}C_2^{-1}\in K
\end{eqnarray*}
    and 
\begin{eqnarray*}
K_8:=S^4K_7S^{-4}=C_3C_5B_{12}B_{15}^{-1}C_8^{-1}C_6^{-1}\in K.
\end{eqnarray*}
It follows $K_7K_8^{-1}$ maps $(a_g,c_1,b_8,b_{11},c_4,c_2)$ to $(a_g,c_1,c_8,b_{11},c_4,c_2)$ that we get
\[
K_9:=(K_7K_8^{-1})K_7(K_7K_8^{-1})^{-1}=A_gC_1C_8B_{11}^{-1}C_4^{-1}C_2^{-1}\in K.
\]
This implies that $K_7K_9^{-1}=B_8C_8^{-1}\in K$ and also
\[
B_1C_1^{-1}=S^{-7}(B_8C_8^{-1})S^7\in K \textrm{ and }B_2C_2^{-1}=S^{-6}(B_8C_8^{-1})S^6\in K.
\]
We finally have the following elements
\begin{eqnarray*}
B_1B_2^{-1}&=&(B_1C_1^{-1})(C_1B_2^{-1}),\\
C_1C_2^{-1}&=&(C_1B_1^{-1})(B_1B_2^{-1})(B_2C_2^{-1})\textrm{ and}\\
A_1A_2^{-1}&=&A_1E_0^{-1}=(A_1B_1^{-1})(B_1B_2^{-1})(B_2E_0^{-1}),
\end{eqnarray*}
which are all contained in $K$. This finishes the proof by Theorem~\ref{thm1}. 
\end{proof}


\section{Torsion generators for $\mod^{*}(\Sigma_{g})$}\label{S4}
Let $R$ be the reflection across the plane $x=0$ of $\Sigma_{g}$ for the model in Figure~\ref{rotation}. Let us consider the orientation reversing diffeomorphism $SR$, denoted by $T$. Note that $T$ satisfies $T(\alpha)= S(\alpha)$ where $\alpha$ is one of the  simple closed curves in  Figure~\ref{curves}, except the curve $d$.

In the proof of the following theorem, we mainly follow the proof of Theorem~\ref{thm1}.

\begin{theorem}\label{thm2}
If $g\geq3$, then $\mod^{*}(\Sigma_{g})$ is generated by the four elements $T$, $A_1A_{2}^{-1}$, $B_1B_{2}^{-1}$, $C_1C_{2}^{-1}$.
\end{theorem}
\begin{proof}
Note that one can generate $\mod^{*}(\Sigma_g)$ by adding an orientation reversing diffeomorphism to the set of Dehn twists
$
\lbrace A_1, A_2, B_1, \ldots, B_g, C_1, \ldots, C_{g-1} \rbrace.
$

Let $G$ be the subgroup of $\mod^{*}(\Sigma_g)$ generated by the set
\[
\lbrace T,  A_1A_{2}^{-1}, B_1B_{2}^{-1}, C_1C_{2}^{-1} \rbrace.
\]
Since $T$ is an orientation reversing element, it is enough to show that the subgroup $G$ contains the Dehn twists
$
 A_1, A_2, B_1, \ldots, B_g, C_1, \ldots, C_{g-1}
$

Observe that $T$ maps the curves $(b_2,b_1)$ to $(b_3,b_2)$ with reverse orientation. Thus,
\[
B_3^{-1}B_2=T(B_2B_1^{-1})T^{-1}\in G, 
\]
which implies that $B_2B_3^{-1}\in G$ by the commutativity.
One can verify that the diffeomorphism $(B_2B_3^{-1})(A_2A_1^{-1})$ satisfies
\[
(B_2B_3^{-1})(A_2A_1^{-1})(b_2,b_3)=(a_2,b_3).
\]
We then obtain,  
\[
A_2B_3^{-1}=\big((B_2B_3^{-1})(A_2A_1^{-1})\big)(B_2B_3^{-1})\big((B_2B_3^{-1})(A_2A_1^{-1})\big)^{-1}\in  G.
\]
From these we get
\begin{eqnarray*}
A_1B_2^{-1}&=&(A_1A_2^{-1})(A_2B_3^{-1})(B_3B_2^{-1})\in G \textrm{ and so}\\
A_1B_1^{-1}&=&(A_1B_2^{-1})(B_2B_1^{-1})\in G.
\end{eqnarray*}
It follows from $T(a_1,b_2)=(c_1,b_3)$ that we get
\[
C_1^{-1}B_3=T(A_1B_2^{-1})T^{-1}\in G.
\]
This implies that $C_1B_3^{-1}\in G$ using the commutativity of the Dehn twists $C_1$ and $B_3$. 
Using this, we also get
\begin{eqnarray*}
C_1B_2^{-1}&=&(C_1B_3^{-1})(B_3B_2^{-1})\in G \textrm{ and so}\\
B_1C_1^{-1}&=&(B_1B_2^{-1})(B_2C_1^{-1})\in G.  
\end{eqnarray*}
The subgroup $G$ contains the following elements:
\begin{eqnarray*}
A_1C_1^{-1}&=&(A_1B_2^{-1})(B_2C_1^{-1}),\\
C_1A_2^{-1}&=&(C_1A_1^{-1})(A_1A_2^{-1}),\\
C_2A_1^{-1}&=&(C_2C_1^{-1})(C_1A_1^{-1}) \textrm{ and}\\
C_2B_1^{-1}&=&(C_2A_1^{-1})(A_1B_1^{-1}).
\end{eqnarray*}
It can be also seen that $T(a_2,c_1)=(e_1,c_2)$ with reverse orientation, which implies that
\[
E_1^{-1}C_2=T(A_1C_1^{-1})T^{-1}\in G,
\]
and also $E_1C_2^{-1}\in G$ by the commutativity.
Consider the diffeomorphism
\[
\Gamma=(B_1C_2^{-1})(C_1C_2^{-1})(A_1C_2^{-1})(B_1C_2^{-1})\in G,
\]
which sends $(e_1,c_2)$ to $(d,c_2)$. This implies that
\[
DC_2^{-1}=\Gamma(E_1C_2)\Gamma^{-1}\in G.
\]
The subgroup $G$ contains 
\[
DA_1^{-1}=(DC_2^{-1})(C_2A_1^{-1}).
\]
By the lantern relation, the following holds:
\[
A_1C_1C_2A_3=A_2E_1D,
\]
where the curves are depicted in Figure~\ref{curves}.  One can rewrite this equation as
\[
A_3=(A_2C_1^{-1})(E_1C_2^{-1})(DA_1^{-1}),
\]
which is clearly contained in $G$. Hence, since 
\[
A_3(B_3B_2^{-1})(a_3)=b_3,
\]
we get 
\[
B_3=\big(A_3(B_3B_2^{-1})\big)A_3\big(A_3(B_3B_2^{-1})\big)^{-1}\in G,
\]
which implies that $B_1,B_2,\ldots, B_{g+1}$ are all contained in $G$ by the action of $T$.  Moreover, we obtain 
\begin{eqnarray*}
C_1&=&(C_1B_1^{-1})B_1\in G, \\
A_1&=&(A_1B_1^{-1})B_1\in G\textrm{ and so}\\
A_2&=&(A_2A_1^{-1})A_1\in G.
\end{eqnarray*}
By conjugating $C_1$ with powers of $T$, we see that $C_1,C_2,\ldots C_{g-1}$ are in $G$. This finishes the proof.
\end{proof}

\begin{theorem}\label{lem11}
If $g\geq 14$, then $\mod^{*}(\Sigma_g)$ can be generated by the two elements $T$ and $T^2A_gB_6A_2E_2^{-1}B_8^{-1}C_1^{-1}$.  The order of $T$ is $g+1$ or $2g+2$  when $g$ is odd or even, respectively.  The order of $T^2A_gB_6A_2E_2^{-1}B_8^{-1}C_1^{-1}$ is $g+1$ or $\frac{g+1}{2}$ when $g$ is even or odd, respectively.
\end{theorem}

\begin{proof}
Let $H$ be the subgroup of $\mod^{*}(\Sigma_g)$ generated by the elements $T$ and $T^2H_1$ so that  $H_1:=A_gB_6A_2E_2^{-1}B_8^{-1}C_1^{-1}$. It is clear that $H_1\in H$. We then get the following element:
\[
H_2:=T^2H_1T^{-2}=C_1B_8E_2E_4^{-1}B_{10}^{-1}C_3^{-1}.
\]
It follows from $H_1H_2^{-1}(a_g,b_6,a_2,e_2,b_8,c_1)=(a_g,e_4,a_2,e_2,b_8,c_1)$ that we have
\[
H_3:=(H_1H_2^{-1})H_1(H_1H_2^{-1})^{-1}=A_gE_4A_2E_2^{-1}B_8^{-1}C_1^{-1}\in H.
\]
Thus, we get
\begin{eqnarray*}
H_1H_3^{-1}&=&B_6E_4^{-1}\in H \textrm{ and so}\\
B_4E_2^{-1}&=&T^{-2}(H_1H_3^{-1})T^2\in H.
\end{eqnarray*}
We also obtain
\begin{eqnarray*}
H_4'&:=&T^5H_1T^{-5}=E_7B_{13}C_6C_4^{-1}B_{11}^{-1}E_5^{-1}\in H \textrm{ and}\\
H_4&:=&(H_4')^{-1}=E_5B_{11}C_4C_6^{-1}B_{13}^{-1}E_7^{-1}\in H.
\end{eqnarray*}
It can be checked that $H_1H_4^{-1}(a_g,b_6,a_2,e_2,b_8,c_1)=(a_g,c_6,a_2,e_2,b_8,c_1)$, which implies
\[
H_5:=(H_1H_4^{-1})H_1(H_1H_4^{-1})^{-1}=A_gC_{6}A_2E_2^{-1}B_{8}^{-1}C_1^{-1}\in H.
\]
Hence,
\begin{eqnarray*}
H_1H_5^{-1}&=&B_6C_6^{-1}\in H \textrm{ and}\\
H_1^{-1}H_5&=&B_6^{-1}C_6\in H.
\end{eqnarray*}
From these, we have
\begin{eqnarray*}
B_g^{-1}A_g&=&T^{-7}(B_6C_6^{-1})T^7\in H,\\
B_2C_2^{-1}&=&T^{-4}(B_6C_6^{-1})T^4\in H \textrm{ and}\\
B_1C_1^{-1}&=&T^{-5}(B_6^{-1}C_6)T^5\in H.\\
\end{eqnarray*}

The subgroup $H$ contains 
\begin{eqnarray*}
H_6:&=&H_1(A_g^{-1}B_g)(E_2B_4^{-1})\\
&=&(A_gB_6A_2E_2^{-1}B_8^{-1}C_1^{-1})(A_g^{-1}B_g)(E_2B_4^{-1})\\
&=&B_gB_6A_2B_4^{-1}B_8^{-1}C_1^{-1}
\end{eqnarray*}
and
\begin{eqnarray*}
H_7:&=&(B_6E_4^{-1})(T^4H_6T^{-4})\\
&=&(B_6E_4^{-1})(B_3B_{10}E_4B_8^{-1}B_{12}^{-1}C_5^{-1})\\
&=&B_3B_{10}B_6B_8^{-1}B_{12}^{-1}C_5^{-1}\in H.
\end{eqnarray*}
Since $H_6H_7^{-1}(b_g,b_6,a_2,b_4,b_8,c_1)=(b_g,c_5,a_2,b_4,b_8,c_1)$, we obtain
\[
H_8:=(H_6H_7^{-1})H_6(H_6H_7^{-1})^{-1}=B_gC_5A_2B_4^{-1}B_8^{-1}C_1^{-1}\in H.
\]
We conclude that
\begin{eqnarray*}
B_6C_5^{-1}&=&H_6H_8^{-1}\in H\textrm{ and}\\
B_6^{-1}C_5&=&H_6^{-1}H_8\in H,
\end{eqnarray*}
which implies that 
\begin{eqnarray*}
C_1B_2^{-1}&=&T^{-4}(C_5B_6^{-1})T^4\in H \textrm{ and}\\
B_1A_1^{-1}&=&T^{-5}(B_6^{-1}C_5)T^5\in H.
\end{eqnarray*}

 Hence, we obtain
\begin{eqnarray*}
C_1C_2^{-1}&=&(C_1B_2^{-1})(B_2C_2^{-1})\in H \textrm{ and}\\
B_1B_2^{-1}&=&(B_1C_1^{-1})(C_1B_2^{-1})\in H.
\end{eqnarray*}
By Theorem~\ref{thm2}, it remains to show that $A_1A_2^{-1}\in H$.
\begin{eqnarray*}
B_2A_2^{-1}&=&T^{-2}(B_4E_2^{-1})T^2 \in H \textrm{ and so}\\
A_1A_2^{-1}&=&(A_1B_1^{-1})(B_1B_2^{-1})(B_2A_2^{-1})\in H,
\end{eqnarray*}
which finishes the proof.
\end{proof}

\begin{theorem}\label{lem12}
If $g\geq 16$, then $\mod^{*}(\Sigma_g)$ is generated by the two elements $T$ and $T^3A_gE_0B_8B_{11}E_3C_2$.
The order of $T$ is $g+1$ or $2g+2$  when $g$ is odd or even, respectively.  The order of $T^3A_gE_0B_8B_{11}E_3C_2$ can be summarized as below:
\[ 
\begin{cases} 
      g+1 & \textrm{if}~ 3\nmid g+1 ~\textrm{and}~ g ~\textrm{is odd,} \\
       2g+2 & \textrm{if}~ 3\nmid g+1 ~\textrm{and}~ g ~\textrm{is even,} \\
      \frac{g+1}{3} & \textrm{if}~ 6\mid g+1, \\
       \frac{2g+2}{3} & \textrm{if}~ 6\mid g+1 ~\textrm{and}~ 6\nmid g+1.
   \end{cases}
\]
\end{theorem}

\begin{proof}
Let $K_1:=A_gE_0B_8B_{11}^{-1}E_3^{-1}C_2^{-1}$ so that the subgroup  
$K$ of $\mod^{*}(\Sigma_g)$ is generated by $T$ and $T^3K_1$. Thus, $K_1\in K$. We have 
\[
K_2:=T^3K_1T^{-3}=C_2^{-1}E_3^{-1}B_{11}^{-1}B_{14}^{-1}E_6^{-1}C_2^{-1}\in K.
\]
Since $K_1K_2^{-1}$ maps the curves $(a_g,e_0,b_8,b_{11},e_3,c_2)$ to $(a_g,e_0,e_6,b_{11},e_3,c_2)$, we obtain the following element:
\[
K_3:=(K_1K_2^{-1})K_1(K_1K_2^{-1})^{-1}=A_gE_0E_6B_{11}E_3C_2\in K.
\]
The subgroup $K$ contains $K_1K_3^{-1}=B_8E_6^{-1}$ and $K_1^{-1}K_3=B_8^{-1}E_6$ which imply that
\begin{eqnarray*}
B_5E_3^{-1}&=&T^{-3}(B_8^{-1}E_6)T^3\in K \textrm{ and so}\\
B_2E_0^{-1}&=&T^{-6}(B_8E_6^{-1})T^6\in K.
\end{eqnarray*}
Then we have 
\begin{eqnarray*}
K_4&:=&(B_5E_3^{-1})K_1=A_gE_0B_8B_{11}B_5C_2\in K \textrm{ and} \\ 
K_5&:=&T^2K_4T^{-2}=C_1E_2B_{10}B_{13}B_7C_4\in K.
\end{eqnarray*}
The diffeomorphism $K_4K_5$ satisfies $K_4K_5(a_g,e_0,b_8,b_{11},b_5,c_2)=(a_g,e_0,b_8,b_{11},c_4,c_2)$, which gives the element
\[
K_6:=(K_4K_5)K_4(K_4K_5)^{-1}=A_gE_0B_8B_{11}C_4C_2\in K.
\]
Hence, $K_4^{-1}K_6=B_5^{-1}C_4$ and $K_4K_6^{-1}=B_5C_4^{-1}$ are contained in $K$. And so, we get
\begin{eqnarray*}
B_2C_1^{-1}&=&T^{-3}(B_5^{-1}C_4)T^3\in K \textrm{ and}\\
B_1A_1^{-1}&=&T^{-4}(B_5C_4^{-1})T^4\in K.
\end{eqnarray*}
Then, we get the elements
\begin{eqnarray*}
    K_7&:=&(C_1B_2^{-1})(B_2E_0^{-1})K_6\\
    &=&(C_1E_0^{-1})A_gE_0B_8B_{11}C_4C_2\\
     &=& A_gC_1B_8B_{11}C_4C_2\in K
\end{eqnarray*}
    and 
\begin{eqnarray*}
K_8:=T^4K_7T^{-4}=C_3C_5B_{12}B_{15}C_8C_6\in K.
\end{eqnarray*}
Since $K_7K_8$ maps $(a_g,c_1,b_8,b_{11},c_4,c_2)$ to $(a_g,c_1,c_8,b_{11},c_4,c_2)$, we have
\[
K_9:=(K_7K_8)K_7(K_7K_8)^{-1}=A_gC_1C_8B_{11}C_4C_2\in K.
\]
We conclude that $K_7K_9^{-1}=B_8C_8^{-1}\in K$ and also $K_7^{-1}K_9=B_8^{-1}C_8\in K$. Furthermore, we have
\[
B_1C_1^{-1}=T^{-7}(B_8^{-1}C_8)T^7\in K \textrm{ and }B_2C_2^{-1}=T^{-6}(B_8C_8^{-1})T^6\in K.
\]
Therefore, the subgroup $K$ contains the following elements
\begin{eqnarray*}
B_1B_2^{-1}&=&(B_1C_1^{-1})(C_1B_2^{-1}),\\
C_1C_2^{-1}&=&(C_1B_1^{-1})(B_1B_2^{-1})(B_2C_2^{-1})\textrm{ and}\\
A_1A_2^{-1}&=&A_1E_0^{-1}=(A_1B_1^{-1})(B_1B_2^{-1})(B_2E_0^{-1}),
\end{eqnarray*}
which finishes the proof by Theorem~\ref{thm2}. 
\end{proof}

\begin{remark}
 We expect that our methods can be further used together with higher powers of $S$ and suitable sets of curves to produce more pairs of torsion generators for the mapping class groups.  Since we could not find a methodical way of expressing our results, we opt not to mention such results in the current paper.   
\end{remark}

\medskip

\noindent
{\bf Acknowledgements.}
This work was supported by the Scientific and Technological Research Council of Turkey (T\"{U}B\.{I}TAK)[grant number 120F118].



\begin{thebibliography}{xxxx}

\bibitem{apy} T. Altun{\"{o}}z, M. Pamuk and O. Yildiz:
\emph{Generating the twist subgroup by involutions}, J. Topol. Anal. {\bf{14}}, (4) (2022), 823--846.

\bibitem{apy1} T. Altun{\"{o}}z, M. Pamuk and O. Yildiz:
\emph{Generating  the extended mapping class group by three involutions}, Osaka J. Math.  {\bf{60}}, (1) (2023), 61--75.

\bibitem{bk} R. \.{I}. Baykur, M. Korkmaz: \emph{ The mapping class group is generated by two commutators}, J. of Algebra {\bf{574}}, (2021), 278--291.

\bibitem{bf} T. E. Brendle, B. Farb:  
\emph{Every mapping class group is generated by $6$ involutions}, J. of Algebra {\bf{278}}, (1) (2004), 187--198.

\bibitem{de} M. Dehn:  
\emph{The group of mapping classes},  In: Papers on Group Theory and Topology. Springer-Verlag, 1987. Translated from the German by J. Stillwell 
(Die Gruppe der Abbildungsklassen, Acta Math {\bf{69}}, (1938), 135--206.

\bibitem{du1} X. Du:
\emph{The extended mapping class group can be generated by two torsions}, Journal of Knot Theory and Its Ramifications {\bf{26}}, (11), (2017).

\bibitem{du2} X. Du:
\emph{The torsion generating set of the extended mapping class groups in low genus cases}, 
Osaka Journal of Mathematics, in press
Osaka J. Math. 58(4): 815-825 (October 2021).

\bibitem{FM} B. Farb and D. Margalit: \emph{A primer on mapping class groups}, Princeton University Press. {\bf{49}} (2011).

\bibitem{G} S. Gervais: \emph{A finite presentation of the mapping class group of a punctured surface}, Topology. {\bf{40}}, (4) (2001), 703--725.

\bibitem{H} S. Humphries:  \emph{ Generators for the mapping class group}, In: Topology of LowDimensional Manifolds, Proc. Second Sussex Conf., Chelwood Gate, (1977), 
Lecture Notes in Math. {\bf{722}}, (2) (1979), Springer-Verlag, 44--47.

\bibitem{iz} 
 I. M. Isaacs,and Thilo Zieschang:\emph{Generating symmetric groups}, Amer. Math. Monthly {\bf{102}}, (8) (1995), 734--739.

\bibitem{j} D. Johnson: \emph{The structure of the Torelli group. {I}. A finite set of generators for $\mathcal{I}$}, Ann. of Math. (2), {\bf{118}}, (3) (1983), 423--442.

\bibitem{ka} M. Kassabov: \emph{Generating mapping class groups by involutions}, ArXiv math.GT/0311455, v1 25Nov2003.

\bibitem{mk2} M. Korkmaz: \emph{Generating the surface mapping class group by two elements}, Trans. Amer. Math. Soc. {\bf{367}}, (8) (2005), 3299--3310.

\bibitem{mk1} M. Korkmaz: \emph{Mapping class group is generated by three involutions}, Math. Res. Lett. {\bf{27}}, (4) (2020), 1095--1108.

\bibitem{l3} W. B. R. Lickorish:  \emph{A finite set of generators for the homeotopy group of a $2$-manifold}, Proc. Cambridge Philos. Soc. {\bf{60}}, (4) (1964), 769--778.

\bibitem{lu} N. Lu: \emph{On the mapping class groups of the closed orientable surfaces}, Topology Proc.  {\bf{13}} (1988), 293–-324.

\bibitem{luo} F. Luo:  \emph{ Torsion elements in the mapping class group of a surface}, ArXiv math.GT/0004048, v1 8Apr2000.

\bibitem{mp} J. D. McCarthy, A. Papadopoulos:  \emph{ Involutions in surface mapping class groups}, Enseign. Math. {\bf{33}}, (2) (1987), 275--290.

\bibitem{m1} N. Monden:  \emph{ Generating the mapping class group of a punctured surface by involutions}, Tokyo J. Math. {\bf{34}}, (2) (2011), 303--312.

\bibitem{m2} N. Monden:  \emph{The mapping class group of a punctured surface is generated by three elements}, Hiroshima Math. J. {\bf{41}}, (1) (2011), 1--9.

\bibitem{m3} N. Monden:  \emph{On minimal generating sets for the mapping class group of a punctured surface}, ArXiv math.GT/2103.01525, v1 2Mar2021.

\bibitem{st} M. Stukow: \emph{ The extended mapping class group is generated by $3$ symmetries}, C. R. Math. Acad. Sci. Paris  {\bf{338}}, (5) (2004), 403--406.

\bibitem{w} B. Wajnryb:  \emph{Mapping class group of a surface is generated by two elements}, Topology {\bf{35}}, (2) (1996), 377--383.

\bibitem{y1} O. Yildiz: \emph{Generating mapping class group by two torsion elements}, Mediterr. J. Math. {\bf{19}}, (59) (2022), 23 pp.

\bibitem{y2} O. Yildiz: \emph{Generating mapping class group by three involutions}, ArXiv math.GT/2002.09151, v1 21Feb2020.


 






\end{thebibliography}
\end{document}